\documentclass[a4paper]{amsart}
\usepackage{amsmath,amsthm,amssymb,hyperref,mathrsfs}
\usepackage{aliascnt}
\usepackage{lmodern}
\usepackage[T1]{fontenc}
\usepackage[textsize=footnotesize]{todonotes}

\usepackage{xcolor}
\definecolor{dblue}{rgb}{0,0,0.70}
\hypersetup{
	unicode=true,
	colorlinks=true,
	citecolor=dblue,
	linkcolor=dblue,
	anchorcolor=dblue
}

\makeatletter
\expandafter\g@addto@macro\csname th@plain\endcsname{%
		\thm@notefont{\bfseries}
	}%
\expandafter\g@addto@macro\csname th@remark\endcsname{%
		\thm@headfont{\bfseries}
	}%
\makeatother

\newtheorem{theorem}
{Theorem}[section]
\newtheorem*{theorem*}{Theorem}

\newaliascnt{lemma}{theorem}
\newtheorem{lemma}[lemma]{Lemma}
\aliascntresetthe{lemma}
\newtheorem*{lemma*}{Lemma}

\newtheorem{claim}[theorem]{Claim}

\newaliascnt{fact}{theorem}

\aliascntresetthe{fact}

\newaliascnt{proposition}{theorem}
\newtheorem{proposition}[proposition]{Proposition}
\aliascntresetthe{proposition}

\newaliascnt{corollary}{theorem}
\newtheorem{corollary}[corollary]{Corollary}
\aliascntresetthe{corollary}

\theoremstyle{remark}

\newaliascnt{remark}{theorem}
\newtheorem{remark}[remark]{Remark}
\aliascntresetthe{remark}
\newaliascnt{question}{theorem}
\newtheorem{question}[question]{Question}
\aliascntresetthe{question}
\newaliascnt{conjecture}{theorem}

\aliascntresetthe{conjecture}

\newtheorem*{question*}{Question}

\newaliascnt{definition}{theorem}
\newtheorem{definition}[definition]{Definition}
\aliascntresetthe{definition}

\newaliascnt{example}{theorem}

\aliascntresetthe{example}

\renewcommand{\restriction}{\mathbin\upharpoonright}

\newcommand{\axiom}[1]{\mathsf{#1}}
\newcommand{\ZFC}{\axiom{ZFC}}
\newcommand{\AC}{\axiom{AC}}

\newcommand{\AD}{\axiom{AD}}
\newcommand{\DC}{\axiom{DC}}
\newcommand{\ZF}{\axiom{ZF}}
\newcommand{\BPI}{\axiom{BPI}}

\newcommand{\GCH}{\axiom{GCH}}

\newcommand{\IS}{\axiom{IS}}
\newcommand{\HS}{\axiom{HS}}
\newcommand{\KWP}{\axiom{KWP}}

\newcommand{\WO}{\axiom{WO}}

\DeclareMathOperator{\dom}{dom}
\DeclareMathOperator{\rng}{rng}
\DeclareMathOperator{\supp}{supp}
\DeclareMathOperator{\rank}{rank}
\DeclareMathOperator{\sym}{sym}

\DeclareMathOperator{\fix}{fix}
\DeclareMathOperator{\fin}{fin}

\DeclareMathOperator{\id}{id}
\DeclareMathOperator{\aut}{Aut}

\DeclareMathOperator{\Add}{Add}

\DeclareMathOperator{\stab}{stab}

\newcommand{\forces}{\mathrel{\Vdash}}

\newcommand{\gaut}[1]{{\textstyle\int_{#1}}}

\newcommand{\power}{\mathcal{P}}

\newcommand{\PP}{\mathbb P}
\newcommand{\QQ}{\mathbb Q}
\newcommand{\RR}{\mathbb R}

\newcommand{\cF}{\mathcal F}
\newcommand{\cG}{\mathcal G}
\newcommand{\cL}{\mathcal L}

\newcommand{\sF}{\mathscr F}
\newcommand{\sG}{\mathscr G}

\newcommand{\tup}[1]{\langle#1\rangle}

\author{Asaf Karagila}
\thanks{The author was supported by the Royal Society grant no.~NF170989 and by UKRI Future Leaders Fellowship MR/T021705/1.}
\email{karagila@math.huji.ac.il}
\urladdr{http://karagila.org}
\address{School of Mathematics,
University of East Anglia.
Norwich, NR4~7TJ, UK
}
\date{July 23, 2021}

\title
{Iterated failures of choice}
\begin{document}
\begin{abstract}
  We combine several folklore observations to provide a working framework for iterating constructions which contradict the axiom of choice. We use this to define a model in which any kind of structural failure must fail with a proper class of counterexamples. For example, the rational numbers have a proper class of non-isomorphic algebraic closures, every partial order embeds into the cardinals of the model, every set is the image of a Dedekind-finite set, every weak choice axiom of the form $\AC_X^Y$ fails with a proper class of counterexamples, every field has a vector space with two linearly independent vectors but without endomorphisms that are not scalar multiplication, etc.
\end{abstract}
\maketitle              
\section{Introduction}
The axiom of choice is a staple of modern mathematics. It was born in controversy, but its incredible usefulness for taming infnite objects and G\"odel's proof that assuming the axiom of choice is not going to introduce a contradiction to mathematics (unless one existed to begin with) made sure that the axiom was accepted into the standard mathematical discourse.

Nevertheless, any and all applications of the axiom of choice are immediately questioned. Is it necessary to use the axiom of choice to prove that every commutative ring has a maximal ideal? Is it necessary for proving Tychonoff's theorem, perhaps just for Hausdorff spaces it is not needed? Etc. These questions are very common, mainly because the axiom of choice asserts the existence of an object, but it does not provide us with means of constructing such an object. Therefore, what we really ask is not whether or not the axiom of choice is necessary, but to what extent are we relying on abstract mathematical rules, and what can we construct by hand in our mathematical work.

These motivations kept the embers of research into weak choice principles and their necessity for various mathematical proofs burning.\footnote{Even when it sometimes seemed that the coals have cooled off completely.} Many of the questions that have been answered so far have been answered by exhibiting a proof that the axiom of choice follows from a certain mathematical theorem (e.g.\ if every vector space has a basis, then the axiom of choice holds \cite{Blass:1984}) or by constructing a model of $\ZF$, a mathematical universe, in which the theorem fails (e.g.\ L\"auchli's construction of a vector space over the rationals which does not have a basis \cite{Lauchli:1963}).

For a simple answer, these will suffice. But we are not mathematicians because we are looking for simple answers. We are mathematicians because we are looking for better questions with more interesting answers. And so, one naturally wants to ask, ``how bad can it get?'' In the basic constructions we usually utilise set theoretic technology to add a counterexample to the universe and we stop there. But these universes will have a reasonable well-behaved extension in which the axiom of choice holds, and therefore the counterexamples disappear. Even in the more general case, where we do add a proper class of counterexamples, we often do so in a way that the universe can still be salvaged, corrected, helped back into a universe of $\ZFC$ in which everything is back to normal again. It is worth specifying that the meaning of well-behaved here is ``with the same ordinals'', since the class of ordinals is in some sense a spine that classifies different models of set theory, so a well-behaved extension is one where the ordinals are unchanged.

The standard set theoretic technology is based on Paul Cohen's method of forcing and it is called \textit{symmetric extensions}. In \cite{Karagila:Iterations} we developed, for the first time, a general framework for iterating symmetric extensions. One of the intended uses for this method was for answering ``how bad can it get'' by showing that ``it can get very \textit{very} bad indeed''. Nevertheless, one key argument was missing from this construction.

In this paper we bridge that gap. We show that by combining several folklore observations we can iterate these sort of structural failures which arise outside of set theory, and we can make the universe very bad. Namely, we cannot restore the axiom of choice, and the universe cannot be saved. In order to restore the axiom of choice, we must add ordinals. This idea is not new, it was developed by Morris in his Ph.D.\ thesis \cite{Morris:PhD}, and while the proof was never published outside of the dissertation itself, it was rediscovered by the author in \cite{Karagila:Morris}. When giving a talk about that very proof in Paris in early November 2019, the stars were aligned correctly, and all these observations came together to create what became a basis for this general technique.

We will give a new proof for a slightly better result than that of Morris, all the while correcting a gap in our previous attempt to do so. We will also define the notion of a ``structural failure of choice'' and show that this method allows us to construct a model in which every structural failure that can be forced occurs within the model, and it does so with a proper class of non-isomorphic counterexamples.

This implies, for example, that for a set $X$ (considered as a discrete space) the nonabelian cohomology $H^1(X,G)$ is not a set for any reasonably interesting group; that each field has a proper class of non-isomorphic vector spaces, all of which admit no linear operators except scalar multiplication; there is a proper class of non-isomorphic algebraic closures of the rational numbers. Indeed, for any set $X$ there is a family of algebraical closures of cardinality $X$; and much more. We also include an interesting historical survey of the result ``every partial order can be embedded into the cardinal structure''. It turned out that this result was also proved in the 1980s by Forti and Honsell, but their paper never received any citations and was quite literally lost in the dunes of time.

Such a mathematical universe, in which all structural failures are witnessed on a proper class, shows the sharp divide between ``concrete'' and ``structural'' mathematics. We can obtain all these structural failures while preserving a well-ordering of the real numbers, or any pre-determined set. In particular that means that the well-ordering of the real numbers, or the Banach--Tarski paradox, or the existence of a Hamel basis for $\ell^\infty$, none of these have any bearing on the generality of structure theories.

This is a point of interest, as many times we do care about a specific field, a specific vector space, a specific object of specific size. And while this will clearly not be enough to imply the axiom of choice, it does show us, for example, how classical analysis is very removed from category theoretic approaches which talk about ``all the groups'' or ``all the $k$-vector spaces'', etc.

\subsection{Structure of this paper}
We begin this paper by covering the technical preliminaries for the construction, these are symmetric extensions and the very basic theory of their iterations. In \autoref{sec:folk} we discuss the folklore results necessary for this work: preservation of initial segments of the universe; creating symmetric ``copies'' of structures from the universe; and the cardinalities of these new ``copies''.

In \autoref{sec:morris} we will describe a slightly improved Morris model and correct a mistake from \cite{Karagila:Morris}, and in \autoref{sec:pandemonium} we will discuss the construction of various ``local'' counterexamples and how to apply the folklore results in order to obtain a general framework for failures.

We finish the main part of the paper with directions for further research, both in set theory, as well as in understanding the standard mathematical structures in a choiceless setting.

\subsection*{Acknowledgements}
The author is indebted to Boban Veli\v{c}kovi\'c for his invitation to Paris and for the opportunity to talk about the Morris model, as well as to the other set theorists of the IMJ--PRG, University of Paris, and Matteo Viale for asking many hard hitting questions and pushing the author to make the needed connections to obtain this general framework. We also thank Jeremy Rickard for bringing to our attention at the right place and the right time a MathOverflow question by the author about the number of algebraic closures a field can have (although it is not answered in full here),\footnote{See the MathOverflow question \href{https://mathoverflow.net/q/325756}{``How many algebraic closures can a field have?''}} and Andr\'es E.~Caicedo, Jason Chen, Hanul Jeon, and David Roberts (who also gave helpful remarks) for suggesting some ideas of global failures of the axiom of choice. And finally to David Asper\'o for making many small suggestions to improve the overall quality of this text.
\section{Preliminaries}

We will use the standard notation of $V_\alpha$ to denote the von Neumann hierarchy, where $\alpha$ is an ordinal. For the readers who are not interested in the set theoretic parts of this paper the key thing to remember is that every $V_\alpha$ is a set, and that every set $x$ appears in some $V_\alpha$. We will also use Greek letters to denote ordinals, with the notable exception of $\pi,\sigma,\tau$ that will be used to denote automorphisms.

Our treatment of forcing is standard, you can find the basics in standard sources such as \cite{Jech:ST}. We say that $\PP$ is a notion of forcing if $\PP$ is a preordered set with a maximum, $1_\PP$. We write $q\leq p$ to mean that $q$ is a stronger condition. Two conditions are compatible if they have a common extension, and are incompatible otherwise. If $\{\dot x_i\mid i\in I\}$ is a collection of $\PP$-names, we write $\{\dot x_i\mid i\in I\}^\bullet$ to denote ``the obvious $\PP$-name'' it creates, namely $\{\tup{1_\PP,\dot x_i}\mid i\in I\}^\bullet$. This notation extends to ordered pairs, sequences, etc. This also somewhat simplifies the canonical names for ground model sets, as we can now write $\check x=\{\check y\mid y\in x\}^\bullet$.

We will use the following group theoretic notion of the wreath product. If $A$ and $B$ are two sets, and $G\subseteq S_A$, $H\subseteq S_B$ (where $S_X$ is the group of all permutations of $X$), then $G\wr H$ is a group of permutations of $A\times B$: $\pi\in G\wr H$, if there are $\pi^*\in G$ and $\tup{\pi_a\mid a\in A}$ such that $\pi_a\in H$, and $\pi(a,b)=\tup{\pi^*(a),\pi_a(b)}$.

\subsection{Symmetric extensions and (some) iterations thereof}

To violate the axiom of choice we cannot use forcing on its own, as forcing preserves the axiom of choice (if present in the ground model). In order to violate the axiom of choice we need to pass from the generic extension, $V[G]$, to an inner model $M$ where it fails. This method is based on the Fraenkel--Mostowski--Specker method of permutation models in the context of $\ZF$ with Atoms.\footnote{See \cite{Jech:AC} for details on the method.} In the case of forcing, a symmetric extension identifies an appropriate class of names which define a model of $\ZF$ where the axiom of choice may fail. Iterating symmetric extensions was developed by the author in \cite{Karagila:Iterations}, and while the full theory is not trivial at all, we will only need a fraction of it here. We start by defining symmetric extensions.

Let $\PP$ be a notion of forcing, and let $\pi$ be an automorphism of $\PP$. We can extend $\pi$ to act on $\PP$-names by recursion,\[\pi\dot x=\{\tup{\pi p,\pi\dot y}\mid\tup{p,\dot y}\in\dot x\}.\]
As the forcing relation is defined from the order of $\PP$, the following lemma should not be surprising. The proof is straightforward and can be found, for example, as Lemma~14.37 in \cite{Jech:ST}.
\begin{lemma*}[The Symmetry Lemma]
$p\forces\varphi(\dot x)\iff\pi p\forces\varphi(\pi\dot x)$.\qed
\end{lemma*}

Fix $\sG\subseteq\aut(\PP)$, and for a name $\dot x$ denote the group $\{\pi\in\sG\mid\pi\dot x=\dot x\}$ by $\sym_\sG(\dot x)$. This is sometimes called the \textit{stabiliser} of $\dot x$. We want to have a way of saying that a name is stable under ``most'' of the automorphisms (in $\sG$). And so we need a suitable notion of a filter.

We say that $\sF$ is a \textit{filter of subgroups of $\sG$} if it is a filter on the lattice of subgroups. That is, it is a non-empty family of subgroups of $\sG$ which is closed under finite intersections and supergroups. We say that $\sF$ is \textit{normal} if for every $H\in\sF$ and every $\pi\in\sG$, $\pi H\pi^{-1}\in\sG$. We will say that $\tup{\PP,\sG,\sF}$ is a \textit{symmetric system} if $\PP$ is a notion of forcing, $\sG$ is a group of automorphisms of $\PP$, and $\sF$ is a normal filter of subgroups of $\sG$. It is easier to assume only the case where $\sF$ is a base for a normal filter, since this is preserved when extending the universe (and perhaps adding new subgroups to $\sG$) and we will do so implicitly.

Fix a symmetric system $\tup{\PP,\sG,\sF}$. If for a $\PP$-name, $\dot x$, we have $\sym_\sG(\dot x)\in\sF$, we say that $\dot x$ is \textit{$\sF$-symmetric}, and if the condition holds hereditarily for names in $\dot x$, we say that it is \textit{hereditarily $\sF$-symmetric}. We write $\HS_\sF$ to denote the class of hereditarily $\sF$-symmetric names.
\begin{theorem*}
Let $\tup{\PP,\sG,\sF}$ be a symmetric system, and let $G\subseteq\PP$ be a $V$-generic filter. The class $M=\HS_\sF^G=\{\dot x^G\mid\dot x\in\HS_\sF\}$ is a transitive class model of $\ZF$ inside $V[G]$ which contains $V$.
\end{theorem*}
This model, $M$, is called a \textit{symmetric extension}. The forcing relation relativises to $\HS_\sF$, namely $p\forces^\HS\varphi(\dot x)$ when $\dot x\in\HS_\sF$ and $p\forces\varphi^\HS(\dot x)$. The usual truth lemma holds for $\forces^\HS$ and the Symmetry Lemma holds as well, assuming $\pi\in\sG$. The proof of the theorem can be found as Lemma~15.51 in \cite{Jech:ST}.

The next step, after taking one symmetric extension, is to take many. This can be done simultaneously by a product of symmetric systems (which is simply a symmetric extension given by the product), or more generally by an iteration.\footnote{A product is a type of iteration also in the ``standard'' context of forcing.} We will not cover the whole apparatus for iterating symmetric extensions in this paper, but give a very informal account of the idea behind it.

If $\tup{\QQ_0,\sG_0,\sF_0}$ is a symmetric system, and $M_0$ is the symmetric extension it defines after fixing some $V$-generic filter $G_0\subseteq\QQ_0$, we want to take a symmetric extension of $M_0$. Say $\tup{\QQ_1,\sG_1,\sF_1}$ is the second symmetric system in $M_0$. By the definition of $M_0$ there is a name $\tup{\dot\QQ_1,\dot\sG_1,\dot\sF_1}^\bullet\in\HS_{\sF_0}$ which is interpreted as the symmetric system, and there is some condition $p\in G_0$ forcing that this is a name of a symmetric system. Our goal is to identify a class of $\QQ_0\ast\dot\QQ_1$-names which will predict the second symmetric extension, as well as understand the conditions necessary for this process to continue in a coherent way.

So when is a $\QQ_0\ast\dot\QQ_1$-name going to be interpreted in this intermediate model? It has to project to a $\QQ_0$-name which is in $\HS_{\sF_0}$ and be forced to be a $\dot\QQ_1$-name that satisfies the property of being in $\HS_{\sF_1}^\bullet$. In particular, that means there is a group $H_0\in\sF_0$ and a name for a group $\dot H_1$ forced to be in $\dot\sF_1$, such that automorphisms in these groups preserve the name at each step.

We can use this to derive a direct definition which looks a bit like that of a symmetric extension. We first observe the following: if $\dot x$ is a $\QQ_0$-name and $p\forces``\dot x$ is a $\dot\QQ_1$-name'', then if $\pi$ is an automorphism, $\pi p\forces``\pi\dot x$ is a $\pi\dot\QQ_1$-name''. This leads us to the following definition, which is a necessary condition for the apparatus to run smoothly.
\begin{definition}
Let $\PP$ be a forcing, $\pi\in\aut(\PP)$ and $\dot A$ a $\PP$-name. We say that $\pi$ \textit{respects} $\dot A$ if $1\forces\pi\dot A=\dot A$. If $\dot A$ has an implicit structure (e.g.\ it is a name for a forcing or a symmetric system) then we also implicitly require that the structure is respected.
\end{definition}
If every $\pi\in\sG_0$ respects the name for $\dot\QQ_1$, then $\tup{q_0,\dot q_1}\mapsto\tup{\pi q_0,\pi\dot q_1}$ is indeed an automorphism of $\QQ_0\ast\dot\QQ_1$. Moreover, if $\dot\pi$ is a name for an automorphism of $\dot\QQ_1$, then $\tup{q_0,\dot q_1}\mapsto\tup{q_0,\dot\pi\dot q_1}$ is also an automorphism of $\QQ_0\ast\dot\QQ_1$.\footnote{Here $\dot\pi\dot q_1$ is a name for a suitable condition that is interpreted, at least below $q_0$, as the result of applying $\dot\pi$ to $\dot q_1$.}

We can therefore combine $\sG_0$ and $\dot\sG_1$ to an automorphism group of $\QQ_0\ast\dot\QQ_1$. To simplify our statements in this section, we will set up the context: $\tup{\QQ_0,\sG_0,\sF_0}$ is a symmetric system and $\tup{\dot\QQ_1,\dot\sG_1,\dot\sF_1}^\bullet\in\HS_{\sF_0}$ is a name forced to be a symmetric system which is also respected by $\sG_0$.

\begin{definition}
If $\tup{\pi_0,\dot\pi_1}$ is a pair such that $\pi_0\in\sG_0$ and $\forces_{\QQ_0}\dot\pi_1\in\dot\sG_1$, then we define the automorphism $\gaut{\tup{\pi_0,\dot\pi_1}}$ of $\QQ_0\ast\dot\QQ_1$ as follows:
\[\gaut{\tup{\pi_0,\dot\pi_1}}\tup{q_0,\dot q_1}=\tup{\pi_0 q_0,\pi_0(\dot\pi_1\dot q_1)}=\tup{\pi_0,\pi_0(\dot\pi_1)(\pi_0\dot q_1)}.\]
We let $\cG_1=\sG_0\ast\dot\sG_1$ denote the group of all such automorphisms, and we call it the \textit{generic semi-direct product}.
\end{definition}

Next we need to handle the filters of groups. If $\dot x$ is a name which is to identify a set in the iterated symmetric extension, we essentially have $H_0\in\sF_0$ and some name $\dot H_1$ forced to be in $\dot\sF_1$ such that whenever $\pi_0\in H_0$ and $\forces_{\QQ_0}\dot\pi_1\in\dot H_1$ (which we will abbreviate as $\forces\tup{\pi_0,\dot\pi_1}\in\tup{H_0,\dot H_1}$), then $\gaut{\tup{\pi_0,\dot\pi_1}}$ respects $\dot x$. We will write $\cF_2=\sF_0\ast\dot\sF_1$ to denote the collection of these pairs which we call \textit{$\cF_2$-supports}.

\begin{definition}
$\tup{H_0,\dot H_1}$ is an $\cF_2$-support of $\dot x$ if whenever $p\forces\tup{\pi_0,\dot\pi_1}\in\tup{H_0,\dot H_1}$, then $p\forces\gaut{\tup{\pi_0,\dot\pi_1}}\dot x=\dot x$. In this case we say that $\dot x$ is \textit{$\cF_2$-respected}, and if the property holds hereditarily we say that it is \textit{hereditarily $\cF_2$-respected}. We denote by $\IS_2$ the class of all hereditarily $\cF_2$-respected names.\footnote{We skipped $\cF_1$ and $\IS_1$ since that is a single-step symmetric extension.}
\end{definition}
\begin{remark}
We will not use the $\gaut{\vec\pi}$ notation in this paper, since we are going to see how it is possible to stitch ``local'' constructions into an iteration. The notation becomes very useful when dealing with arbitrary iterations and it is used extensively in \cite{Karagila:Iterations}, so it is worth introducing, to help and ease the reader into the subject.
\end{remark}

We can now extend this to any length with a finite support iteration.
\begin{definition}
Suppose that $\tup{\dot\QQ_\alpha,\dot\sG_\alpha,\dot\sF_\alpha\mid\alpha<\delta}$ and $\tup{\PP_\alpha,\cG_\alpha,\cF_\alpha\mid\alpha\leq\delta}$ are sequences satisfying the following:
\begin{enumerate}
\item $\PP_\alpha$ is the finite support iteration of $\dot\QQ_\alpha$.
\item $\forces_\alpha\tup{\dot\QQ_\alpha,\dot\sG_\alpha,\dot\sF_\alpha}^\bullet$ is a symmetric system.
\item $\cG_0=\{\id\}$, $\cG_{\alpha+1}=\cG_\alpha\ast\dot\sG_\alpha$, and if $\alpha$ is a limit, then $\cG_\alpha$ is the direct limit of $\cG_\beta$ for $\beta<\alpha$.
\item $\cF_0=\{\cG_0\}$, $\cF_{\alpha+1}=\cF_\alpha\ast\dot\sF_\alpha$, and if $\alpha$ is limit, then $\cF_\alpha$ is the direct limit of $\cF_\beta$ for $\beta<\alpha$, i.e.\ the collection of sequences $\tup{\dot H_\beta\mid\beta<\alpha}$ such that $\forces_\beta\dot H_\beta\in\dot\sF_\beta$ and for all but finitely many $\beta<\alpha$, $\forces_\beta\dot H_\beta=\dot\sG_\beta$.
\item For each $\alpha$, $\tup{\dot\QQ_\alpha,\dot\sG_\alpha,\dot\sF_\alpha}^\bullet$ is hereditarily $\cF_\alpha$-respected, and the name itself is respected by any automorphism in $\cG_\alpha$.
\end{enumerate}
Then the class $\IS_\alpha$ of hereditarily $\cF_\alpha$-respected names, for $\alpha\leq\delta$, is the class of $\PP_\alpha$-names which predicts the iteration of symmetric extensions. Moreover, $\IS_{\alpha+1}$ will be a symmetric extension of $\IS_\alpha$ by $\tup{\dot\QQ_\alpha,\dot\sG_\alpha,\dot\sF_\alpha}^\bullet$.
\end{definition}
We also have a relativised forcing relation, $\forces^\IS$, which is defined similarly as $\forces^\HS$.
\begin{theorem}[Theorem~9.1 in \cite{Karagila:Iterations}]\label{thm:preservation}
Let $\tup{\dot\QQ_\alpha,\dot\sG_\alpha,\dot\sF_\alpha\mid\alpha<\delta}$ be a symmetric iteration such that for all $\alpha$, $\forces_\alpha``\dot\sG_\alpha$ witnesses the homogeneity of $\dot\QQ_\alpha$'' and let $\eta$ be some ordinal such that there is $\alpha<\delta$ such that for all $\beta\in[\alpha,\delta)$,\[\forces_\beta^\IS\tup{\dot\QQ_\beta,\dot\sG_\beta,\dot\sF_\beta}^\bullet\text{ does not add new sets of rank }<\check\eta,\] and let $G\subseteq\PP_\delta$ be a $V$-generic filter. Then $V_\eta^{\IS_\alpha^{G\restriction\alpha}}=V_\eta^{\IS_\delta^G}$.\qed
\end{theorem}
In other words, if each symmetric extension is homogeneous, and we do not add sets of rank $\eta$ on a tail below $\delta$, then we do not add such sets at the $\delta$th stage either. This is very important, as non-trivial forcing will tend to add Cohen reals at limit steps, or even collapse cardinals if we are not careful about our chain conditions.

Almost as importantly, this means that if we guarantee increasing distributivity and homogeneity, then we may iterate even a class length iteration, while preserving $\ZF$ in the final model. (This is Theorem~9.2 in \cite{Karagila:Iterations}.)
\begin{remark}
  We diverge from \cite{Karagila:Iterations} in the definition of supports, as we do not discuss where the names $\dot H_\alpha$ come from. In this paper we hint that these are $\PP_\alpha$-names, whereas in \cite{Karagila:Iterations} we only require them to be $\PP_\delta$-names. This is fine due to the finite support nature of the iteration, so the two definitions are equivalent. We point out in \cite{Karagila:Iterations} that we utilise mixing over antichains to define the automorphisms and the iteration anyway, and there is something to be gained by allowing $\dot H_\alpha$ to be, in fact, a $\PP_\delta$-name for a member of $\dot\sF_\alpha$. In other words, we are allowed to hold off on choosing our pointwise groups until much later in the iteration.

  This has a certain elegance to it, and it is certainly useful in smoothing out the definition (although causing bumps elsewhere). Nevertheless, we do not really care for this here, since our situation is going to be quite specific.
\end{remark}
\section{Combining folklore results}\label{sec:folk}
\subsection{Rank-preserving forcings}
Call a permutation group $\sG$ of a set $I$ \textit{finitely pacifying} if for every finite sets $E,F,F'\subseteq I$ such that $E\cap F'=\varnothing$ there is some $\pi\in\sG$ such that $\pi\restriction E=\id$ and $\pi''(F\setminus E)\cap F'=\varnothing$.
\begin{lemma}\label{lemma:folk-distrib}
  Let $V$ be a model of $\ZF$, and suppose that $\QQ$ in $V$ is some forcing whose finite products do not add any sets of rank $<\alpha$. Let $\PP$ be the finite support product $\prod_{i\in I}\QQ$. Let $\sG$ be a finitely pacifying permutation group of $I$ applied to $\PP$ by means of permuting the index set, i.e. $\pi\tup{q_i\mid i\in I}=\tup{q_{\pi i}\mid i\in I}$. Finally, let $\sF$ be the filter of subgroups generated by $\{\fix(E)\mid E\in[I]^{<\omega}\}$, where $\fix(E)$ denotes $\{\pi\in\sG\mid\pi\restriction E=\id\}$.

  Then $\forces^\HS\dot V_\alpha=\check V_\alpha$. In other words, no new sets of rank $<\alpha$ are added to the symmetric extension given by $\tup{\PP,\sG,\sF}$.
\end{lemma}
\begin{proof}
  We prove this by induction on the rank of $\dot x$.\footnote{By rank we mean either as a set, or more appropriately as a $\PP$-name.} Suppose that $\dot x\in\HS$ and $\forces^\HS\rank(\dot x)<\check\alpha$. By the induction hypothesis, we may replace $\dot x$ by the following name, \[\{\tup{p,\check y}\mid p\forces\check y\in\dot x\},\] To see that we can indeed do that note that if $p\forces\dot y\in\dot x$, then by the induction hypothesis $p$ must force that $\dot y$ is equal to a ground model element, so by extending $p$ if necessary we may assume that $\dot y$ was already $\check y$ for some $y\in V_\alpha$.

  We also note that if $\pi\in\sym(\dot x)$, then $\pi p\forces\pi\check y\in\pi\dot x$ is just $\pi p\forces\check y\in\dot x$. Therefore, $\pi$ also preserves the new name above. So indeed, we may assume that $\dot x$ was exactly that name, which is in $\HS$ after all. Let $E$ denote some finite set such that $\fix(E)\subseteq\sym(\dot x)$.

  Now let $p$ be any condition in $\PP$ such that $p\forces\check y\in\dot x$. Then, we claim, by the finitely pacifying condition on $\sG$, $p\restriction E\forces\check y\in\dot x$. Indeed, if $q\leq p\restriction E$, let $F$ denote $\supp p$ and $F'=\supp q$, then we can apply the finitely pacifying condition to $E,F,F'$ and obtain $\pi\in\fix(E)$ such that $\pi q$ is compatible with $p$. Therefore $q\forces\check y\in\dot x$ as well.

  This means that we can replace $\dot x$ by $\{\tup{p\restriction E,\check y}\mid p\forces\check y\in\dot x\}$, which is actually a name in the finite product $\prod_{i\in E}\QQ$, and by the assumption on $\QQ$ must be equal to some ground model element as wanted.
\end{proof}

To utilise this lemma we need to show that we can find such $\QQ$. That is not obviously something we can do over any model of $\ZF$. It is not clear, for example, that over Gitik's model \cite{Gitik:1980,Gitik:1985} any forcing is at all $\sigma$-distributive. However, in our context we are iterating symmetric extensions starting from a model of $\ZFC$, which we can use for our advantage.

\begin{theorem}\label{thm:add-k-dist}
Let $V$ be a model of $\ZFC$, let $\PP$ be a symmetric iteration of any ordinal length, and let $G\subseteq\PP$ be a $V$-generic filter. Then there is some regular $\kappa$ such that $\Add(\kappa,1)^V$ satisfies the conditions of \autoref{lemma:folk-distrib} in $\IS^G$ for any $\alpha$.
\end{theorem}
\begin{proof}
  Let $\kappa$ be any regular cardinal greater or equal to $|\PP|^+$ in $V$. Then, in $V$, $\PP$ satisfies the $\kappa$-c.c.\ while $\Add(\kappa,1)^V$ is $\kappa$-closed. Therefore in the full generic extension by $\PP$, $\Add(\kappa,1)^V$ is still $\kappa$-distributive, and moreover it is isomorphic to any of its finite products. Taking $\kappa$ large enough such that $V[G]\models\kappa>|V_\alpha^{\IS^G}|$ will satisfy our conditions.

  To see that, suppose that $\{D_i\mid i\in V_\alpha^{\IS^G}\}$ is a family of dense open subsets of $\Add(\kappa,1)^V$ which is in $\IS^G$ Then in $V[G]$ this family has size strictly smaller than $\kappa$ and therefore its intersection is dense. Density is absolute to $\IS^G$ and therefore the distributivity holds.
\end{proof}
We can now generalise the notion of \textit{$J$-pacifying} by considering an ideal $J$ on the index set $I$ which contains the finite sets, and requiring that $E,F,F\in J$. We can then talk about the product having $J$-support, and the results will stay the same. For example, if one considers a $\kappa$-support product from $V$ in $V[G]$, then we have a natural ideal associated with the product in $V[G]$ to which we can apply these results.
\subsection{Symmetric copies}
Our second folklore result concerns one of the standard methods in the study of the axiom of choice, at least when symmetric extensions are involved. We often want to add a ``generic copy'' of a ground model structure and use symmetries to control which subsets of the copy are preserved in the symmetric extension, and thus we can call this structure a ``symmetric copy''. One can arrange, for example, a vector space which is not generated by any finite set, but whose proper subspaces are all finitely generated by simply making sure to remove all the subsets which generate proper subspaces that are not finite dimensional.\footnote{See \cite{Karagila:MSc} for an example of this approach. The original result is due to L\"auchli, \cite{Lauchli:1963} in the context of permutation models or the construction in Theorem~10.11 in \cite{Jech:AC} for a modern description of the model.}

Fix a language $\cL$ and let $M$ be a $\cL$-structure. We let $\sG_M$ be an automorphism group of $M$ and we let $\sF_M$ denote a normal filter of subgroups generated by $\{\fix(A)\mid A\in I\}$, where $I$ is some ideal of subsets of $M$. We say that $X$ is stable under $\pi\in\sG_M$ if $\pi``X=X$ and that it is stable under some $H\subseteq\sG_M$ if it is stable under every $\pi\in H$. We let $\stab_{\sG_M}(X)$ denote the largest group under which $X$ is stable, and we will say that $X$ is $\sF_M$-stable if $\stab_{\sG_M}(X)\in\sF_M$. This notion extends naturally to relations on $M$. We will always assume that the stabilisers of all finite subsets of $M$ are in $\sF_M$.
\begin{theorem}\label{thm:sym-copy}
Suppose that $M$ is a $\cL$-structure, $\sG_M$ and $\sF_M$ are as above. Then there is a symmetric extension of $V$ in which there is a symmetric copy of $M$ whose only subsets (and relations) are the copies of $\sF_M$-stable subsets from $V$.
\end{theorem}
This type of theorem was proved many times over from the early advents of symmetric extensions (see for example Plotkin in \cite{Plotkin:1969}, Hodges in \cite{Hodges:1974}, and more recently a very similar theorem in the author's paper with Noah Schweber \cite{Karagila:Schweber}). The idea is simple: for each $m\in M$ add an infinite set of $\kappa$-Cohen subsets, for a suitable $\kappa$, and consider the permutation group $\sG_M\wr S_\kappa$ with the filter generated by $\sF_M\times\sF_\kappa$, here $\sF_\kappa$ is the normal filter of subgroups generated by fixing pointwise a set in $[\kappa]^{<\kappa}$.
\begin{proof}
  Let $\kappa\geq|M|^+$ be a regular cardinal, let $\PP=\Add(\kappa,M\times\kappa)$, let $\sG=\sG_M\wr S_\kappa$, and let $\sF=\{\fix(A\times B)\mid A\text{ is }\sF_M\text{-stable}, B\in[\kappa]^{<\kappa}\}$.

  Recall that $p\in\PP$ is a function whose domain is a subset of $M\times\kappa\times\kappa$ such that $|p|<\kappa$, and that $\pi\in\sG$ is composed of $\pi^*\in\sG_M$ and $\tup{\pi_m\mid m\in M}$ such that $\pi_m\in S_\kappa$, and its action on $\PP$ is the standard action in these constructions: \[\pi p(\pi(m,\alpha),\beta)=\pi p(\pi^*m,\pi_m\alpha,\beta)=p(m,\alpha,\beta).\]

  We define the following names.
  \begin{enumerate}
  \item For $m\in M,\alpha<\kappa$ let $\dot x_{m,\alpha}=\{\tup{p,\check\beta}\mid p(m,\alpha,\beta)=1\}$.
  \item For $m\in M$ let $\dot a_m=\{\dot x_{m,\alpha}\mid\alpha<\kappa\}^\bullet$.
  \item Let $\dot N=\{\dot a_m\mid m\in M\}^\bullet$.
  \item For $R\subseteq M^n$ we let $\dot R_N=\{\tup{\dot a_{m_0},\dots,\dot a_{m_{n-1}}}^\bullet\mid\tup{m_0,\dots,m_{n-1}}\in R\}^\bullet$.
  \end{enumerate}

  It is a straightforward check to see that $\pi\dot x_{m,\alpha}=\dot x_{\pi^*m,\pi_m\alpha}$, consequently we have that $\pi\dot a_m=\dot a_{\pi^*m}$. Suppose that $R\subseteq M$ such that $\stab_{\sG_M}(R)\in\sF_M$,\footnote{We can assume $R\subseteq M^n$, but taking $n=1$ simplifies the notation.} then $\dot R_N\in\HS$. To see this, note that if $\pi\in\sG$ such that $\pi^*\in\stab_{\sG_M}(R)$, then we have that: \[\pi\dot R_N=\{\pi\dot a_m\mid m\in R\}^\bullet=\{\dot a_{\pi^*m}\mid m\in R\}^\bullet=\{\dot a_m\mid m\in R\}^\bullet=\dot R_N.\] Therefore, taking any $E\in[\kappa]^{<\kappa}$ will satisfy that $\stab_{\sG_M}(R)\times\fix(E)\subseteq\sym(\dot R)$.

  In the other direction, if $\dot R\in\HS$ and $p\forces\dot R\subseteq\dot N$, since $\PP$ does not add subsets to $M$, there is some $q\leq p$ and some $R'\subseteq M$ such that $q\forces\dot R=\dot R'_N$. Our goal is to show that $\{\pi^*\mid\pi\in\sym(\dot R)\}\subseteq\stab_{\sG_M}(R')$, and therefore $\stab_{\sG_M}(R')\in\sF_M$.

  Since $q\forces\dot R=\dot R'_N$ we have that $q\forces\dot a_m\in\dot R$ if and only if $m\in R'$. Suppose that $\pi\in\sym(\dot R)$, using the Homogeneity Lemma below we can find $\tau\in\sym(\dot R)$ such that $\tau^*=\id$, such that $\tau\pi q$ is compatible with $q$, and $\tau\pi q\leq p$. Since $\tau\pi q\forces\dot a_{\pi^*m}\in\dot R$ and it is compatible with $q$, it must be that $q\forces\dot a_{\pi^*m}\in\dot R$ as well, since $q$ already decides all the statements of this form. Therefore $\pi^*m\in R'$. By applying the same argument to $\pi^{-1}$ we obtain that indeed $\pi^*``R'=R'$.

  Therefore in the symmetric extension it must be the case that $\dot R$ is interpreted as a copy of an $\sF_M$-stable set.
\end{proof}
  \begin{lemma}[Homogeneity Lemma]
    Suppose that $\tau\in\sG_M$ is any automorphism and $p,q\in\PP$ are any two conditions, then there is $\pi\in\sG$ such that $\pi^*=\tau$ and $\pi p$ is compatible with $q$.
  \end{lemma}
\begin{proof}
  It is enough to prove this under the assumption that $\tau=\id$, since in that case if we have $\pi$ such that $\pi^*=\tau$ and $\pi p$ is incompatible with $q$ we can find some $\sigma\in\sG$ such that $\sigma^*=\id$ and $\sigma\pi p$ is compatible with $q$, but since $(\sigma\pi)^*=\sigma^*\pi^*=\pi^*$ the conclusion follows.

  Suppose now that we are given $p$ and $q$. For each $m\in M$, define $\alpha_p$ to be $\sup\{\alpha+1<\kappa\mid\exists m:\tup{m,\alpha}\in\dom p\}$, define $\alpha_q$ similarly, and let $\alpha=\max\{\alpha_p,\alpha_q\}$. Note that $\kappa$ is regular and $\alpha<\kappa$ since $|M|<\kappa$, so the permutation $\pi$ such that $\pi^*=\id$ and $\pi_m$ is the permutation which switches the two intervals $[0,\alpha)$ with $[\alpha,\alpha+\alpha)$, for all $m\in M$, is in $\sG$ and therefore $\dom \pi p\cap\dom q=\varnothing$.\footnote{We may choose $\pi_m$ to be some other appropriate automorphism on each coordinate, rather than uniformly move all the coordinates at once.}
\end{proof}
It is worth noting here that the Homogeneity Lemma is essentially the proof that $\sG_M\wr\sG_\kappa$ is pacifying, which we will need for applying the results from the previous part.
\subsection{New cardinalities in products of symmetric extensions}
Suppose that $M_0$ and $M_1$ are two structures in the ground model (not necessarily in the same signature, but not necessarily different either). And let $\tup{\QQ_0,\sG_0,\sF_0}$ and $\tup{\QQ_1,\sG_1,\sF_1}$ be two symmetric systems as in the construction of \autoref{thm:sym-copy}. We want to understand the relationship between $N_0$ and $N_1$, the symmetric copies of $M_0$ and $M_1$ respectively, in the symmetric extension given by the product of the two symmetric systems.

\begin{theorem}\label{thm:incomp}
Assume that $M_1$ is not fixed pointwise by a group in $\sF_1$, then there is no surjection from $N_0$ to $N_1$ in the symmetric extension given by the aforementioned product. Therefore there are no injections either, and the cardinals $|N_0|$ and $|N_1|$ are incomparable.
\end{theorem}
\begin{proof}
  Let $\dot f\in\HS$ and $\tup{p_0,p_1}$ be such that $\tup{p_0,p_1}\forces\dot f\colon\dot N_0\to\dot N_1$ is a function. Let $\tup{H_0,H_1}\in\sF_0\times\sF_1$ be such that $H_0\times H_1\subseteq\sym(\dot f)$, and without loss of generality we may assume that $H_i$ fixes $p_i$.

  Let $\dot a_{m_1}$ be a name in $\dot N_1$ such that some $\pi_1\in H_1$ satisfies $\pi_1\dot a_{m_1}\neq\dot a_{m_1}$. Such $m_1$ exists since $M_1$ is not fixed pointwise by any group in $\sF_1$. Let $\tup{q_0,q_1}\leq\tup{p_0,p_1}$ such that for some $m_0\in M_0$, $\tup{q_0,q_1}\forces\dot f(\dot a_{m_0})=\dot a_{m_1}$.

  Applying the Homogeneity lemma to $H_1$ and $q_1$ we can obtain some $\tau_1\in H_1$ such that $\tau_1\pi_1 q_1$ is compatible with $q_1$. Therefore we get that \[\tau\pi\tup{q_0,q_1}=\tup{q_0,\tau_1\pi_1 q_1}\forces\dot f(\dot a_{m_0})=\tau\pi\dot a_{m_1}=\pi\dot a_{m_1}\neq\dot a_{m_1}.\] However, since $q_1$ and $\tau_1\pi_1 q_1$ are compatible, so must be $\tau\pi\tup{q_0,q_1}$ and $\tup{q_0,q_1}$. In particular, it follows that $\tup{q_0,q_1}$ could not force that $\dot f$ is a function to begin with, and so $\tup{p_0,p_1}$ could not have done that either.
\end{proof}

The argument extends, of course, to any finite product, and much more. This argument is the same as the one in the proof of Theorem~3.5 in \cite{Karagila:Embeddings}. What we see here is that any structure which has a reasonably rich automorphism group can be used to create new cardinals in the symmetric extension. Seeing how we are mostly interested in structures that have very rich automorphisms, the assumptions of this theorem will be satisfied in all of its applications moving forward.
\section{First application: a slightly better Morris model}\label{sec:morris}
In \cite{Karagila:Morris} we construct, to a certain degree, the model Douglass B.~Morris constructed in his Ph.D.\ thesis \cite{Morris:PhD}. In this model, for every ordinal $\alpha$ there is a set $A_\alpha$ which is a countable union of countable sets, but $\power(A_\alpha)$ can be mapped onto $\omega_\alpha$.

Morris' original proof was never published beyond his Ph.D.\ thesis and a small announcement in the Notices of the American Mathematical Society,\footnote{Notices of the American Mathematical Society \textbf{17} (1970), 70T-E27.} but knowledge of the result was preserved due to its appearance as Problem~5.14 in \cite{Jech:AC}. In \cite{Karagila:Morris} a slightly simplified version of the model was constructed by the author using the mechanisms of iterations of symmetric extensions that allowed the construction to be done by dealing locally with each $\omega_\alpha$ and then combining the results immediately.\footnote{For an overview of Morris' original proof and how it differs from the author's see \cite[\S4.1]{Karagila:Morris}.}

One of the natural questions when confronted with such a result is ``can we replace $\omega_\alpha$ by $V_\alpha$'', since that would imply that not only every ordinal is the image of the power set of a set which is a countable union of countable sets, but indeed that \textit{every} set is such image. This can be done using the observations we include in \autoref{sec:folk} and suitably modifying the local constructions of \cite{Karagila:Morris}. We will present here a slightly improved version of the construction, which will also remove a small gap in the proof appearing in \cite{Karagila:Morris}.

Specifically, the model is constructed by taking a product of two-step iterations of symmetric extensions. Focusing on a single ``local copy'' we start by adding a symmetric copy of a countable union of countable sets. Namely, consider the structure $\omega\times\omega$ with the group of automorphisms which permute each $\{n\}\times\omega$ separately,\footnote{We may assume that these permutations only move finitely many points overall, but with the advents of iterations of symmetric extensions it is not a particularly important assumption.} and with the filter of groups given by fixing $E\times\omega$ pointwise for some finite $E\subseteq\omega$.

After adding our set $A$ which is a countable union of countable sets but is not countable, we let $T$ be the forcing which chooses a single point from each of the countable sets in the union, this adds a subset to $A$, and in fact this is a symmetric copy of $\omega^{<\omega}$ which has no branches.

The second symmetric extension, as presented in \cite{Karagila:Morris}, is given by generically adding an almost disjoint family of branches to $T$ indexed by $\kappa\times\omega$ by considering pairs $\tup{t,f_t}$ such that $t$ is a function from $E\in[\kappa\times\omega]^{<\omega}$ into $T$ and $f_t\colon\power(E)\to\omega$, such that if $f_t(a)=n$, then all the conditions in $t\restriction a$ are pairwise disjoint above $n$. We will improve this step of the construction here.

We then partition the generic sequence of branches, now indexed by $\kappa\times\omega$, into a $\kappa$-sequence of $\omega$-blocks and allow permutations to move each $\omega$-block separately without changing the $\kappa$-sequence. This makes the set of branches a Dedekind-finite set which can be mapped onto $\kappa$.

The model constructed by the author is using the fact that at each of these local copies we interact only with our local $\kappa$, and so we can take the product of these two-step iterations and utilise what we already know about iterations of symmetric extensions to obtain the ``simplified Morris model''.\footnote{Here we utilise the fact that each step is $\kappa$-closed iterated by a c.c.c.\ forcing; Morris' original iteration was using c.c.c.\ forcings in a more complicated manner which requires a careful attention to the details at limit steps. Hence the origin of the term ``simplified''.}

It seems easy to answer the question we began with. To get a surjection onto $V_\kappa$, we simply need to add $V_\kappa\times\omega$ branches to $T$ instead of $\kappa\times\omega$. This, however, means that we are no longer working separately on each $\omega_{\alpha+1}$, but rather at each step we need to take into accounts all the previous steps. While it is certainly not impossible,\footnote{Morris' original proof does just that.} it does complicate things. For one, it is not immediately clear why the forcing remains distributive until one notes that due to \autoref{thm:add-k-dist} no new subsets of small rank are added, assuming that $\kappa$ was chosen correctly. The rest of this section will be devoted to the details of this proof. As some of these proofs will be very similar to those appearing in \cite{Karagila:Morris}, we will sometimes choose to sketch the proofs of certain claims, giving only the idea behind the proof instead.
\subsection{Morris' model revisited}
We define an iteration of symmetric extensions by recursion, working in $V$, a model of $\ZFC+\GCH$. We denote by $\PP_\alpha,\cG_\alpha$ and $\cF_\alpha$ the symmetric iteration up to $\alpha$. This is a finite support iteration.

Suppose that $\PP_\alpha,\cG_\alpha$ and $\cF_\alpha$ were defined. We will define $\dot\QQ_\alpha,\dot\sG_\alpha,\dot\sF_\alpha$, which is itself a two-step iteration, such that the following is true:
\begin{enumerate}
\item $\forces_\alpha^\IS\tup{\dot\QQ_\alpha,\dot\sG_\alpha,\dot\sF_\alpha}^\bullet$ does not add sets of rank $\check\alpha$.
\item $\forces_{\alpha+1}^\IS$ There exists a set $\dot A_\alpha$ which is a countable union of countable sets, and $\power(\dot A_\alpha)$ maps onto $\dot V_\alpha$.
\end{enumerate}

Let $\kappa$ be a suitable cardinal in $V$ for the $\alpha$th stage. Namely, one for which $|\dot V_\alpha^{\IS_\alpha}|<\kappa$ so that the conditions of \autoref{thm:add-k-dist} hold. The symmetric system $\tup{\QQ_{\alpha,0},\dot\sG_{\alpha,0},\dot\sF_{\alpha,0}}$ is simply the symmetric system obtained from \autoref{thm:sym-copy}, with $\kappa$ being this suitable cardinal, applied to the structure $M=\omega\times\omega$ with $\sG_M=\{\id\}\wr S_\omega$, in other words these are permutations of $\omega\times\omega$ such that whenever $\pi(n,m)=\tup{n',m'}$, $n=n'$. The filter $\sF_M$ is generated by groups of the form $\fix(E)$ for $E\in[\omega]^{<\omega}$, where $\fix(E)$ is $\{\pi\in\sG_{\alpha,0}\mid\pi\restriction E\times\omega=\id\}$. We will write $\forces_{\alpha.5}$ to denote the forcing relation of $\PP_\alpha\ast\dot\QQ_{\alpha,0}$, and similarly $\IS_{\alpha.5}$ for the class of names predicting the symmetric extension by $\QQ_{\alpha,0}$, etc.

Let $\dot A_\alpha$ denote the symmetric copy of $\omega\times\omega$ added by the symmetric system, and let $\dot A_{\alpha,n}$ denote the symmetric copy of $\{n\}\times\omega$. It is not hard to verify that indeed $\forces_{\alpha.5}^\IS``\dot A_{\alpha,n}$ is countable'' and that $\tup{\dot A_{\alpha,n}\mid n<\omega}^\bullet\in\IS_{\alpha.5}$.

\begin{claim}
  $\forces_{\alpha.5}^\IS|\dot A_\alpha|\neq\aleph_0$.
\end{claim}
\begin{proof}[Sketch of Proof]
Suppose that $\dot f\in\IS_{\alpha.5}$ and $\tup{p,\dot q}\forces^\IS_{\alpha.5}\dot f\colon\check\omega\to\dot A$. Suppose that $E$ is a finite subset of $\omega$ such that the last group in the $\cF_{\alpha.5}$-support of $\dot f$ is $\fix(E)$. Simply extend $\dot q$ to some $\dot q'$ which decides that $\dot f(\check n)\in\dot A_{\alpha,k}$ for some $k\notin A$, then use the Homogeneity Lemma to show that this is impossible. Therefore $\tup{p,\dot q}\forces^\IS_{\alpha.5}\rng\dot f\subseteq\bigcup_{k\in E}\dot A_{\alpha,k}$.
\end{proof}

Let $T$ denote the choice-tree from the family $\{A_{\alpha,n}\mid n<\omega\}$. In other words, $T=\bigcup_{n<\omega}\prod_{k<n}A_{\alpha,k}$. The above claim shows that $T$ has no branches. As we said, we wish to add many new branches to $T$, but we will need to ensure that they are sufficiently distant from one another, otherwise we can't control them easily with permutations from $\sG_{\alpha,0}$, but not too disjoint; otherwise we can use them to encode an enumeration of $A$, which ultimately leads to the collapse of cardinals.

We next define $\dot\QQ_{\alpha,1}$, and this is where we diverge from \cite{Karagila:Morris}. For simplicity, let us work in the extension given by $\IS_{\alpha.5}$ and come up with names for all the objects later on. The conditions in $\QQ_{\alpha,1}$ are going to be finite sequences $t=\tup{t_i\mid i\in E}$, where $E\subseteq V_\alpha\times\omega$ is a finite set, and $t_i\in T$ for all $i\in E$, we say that $E$ is the \textit{support} of $t$. The order is not defined just as pointwise extensions, but instead we say that $t\leq_{\QQ_{\alpha,1}}s$ if the following conditions hold:
\begin{enumerate}
\item $\supp s\subseteq\supp t$ and $s_i\subseteq t_i$ for all $i\in\supp s$.
\item For all $i,j\in\supp s$, if $t_i(k)=t_j(k)$, then $k\in\dom s_i\cap\dom s_j$.
\end{enumerate}
In other words, $t$ extends $s$ if not only the natural extension condition holds, but on $\supp s$ the extensions are pairwise disjoint.

We say that $t$ is a \textit{minimal condition} if whenever $s$ is such that $t\leq s$ and $\supp t=\supp s$ we have that $t=s$. So, a minimal condition means that we cannot trim any stems in $t$ without removing them completely. We will use the terminology $k$-minimal to mean that the minimality condition only affects stems whose length is greater than $k$.

We let $\sG_{\alpha,1}$ be the group $\{\id\}\wr S_\omega^{\fin}$, where $S_\omega^{\fin}$ is the group of finitary permutations of $\omega$, acting on $V_\alpha\times\omega$. In other words, we consider permutations of $V_\alpha\times\omega$ so that when we consider the partition $\{\{x\}\times\omega\mid x\in V_\alpha\}$, each cell is stabilised, and within each cell only finitely many points are moved. Finally, $\sF_{\alpha,1}$ is defined by pointwise stabilisers of finite subsets of $V_\alpha\times\omega$.\footnote{Note that this time we truly mean finite, and not finite union of $\omega$-blocks like we did in the case of $\sF_{\alpha,0}$.}

Define now a name for the $\tup{x,n}$th branch added by $\QQ_{\alpha,1}$, \[\dot b_{x,n}=\{\tup{t,\check a}\mid\tup{x,n}\in\dom t, t_{x,n}(i)=a\}.\]
It is easy to see that if $\pi\in\sG_{\alpha,1}$, then $\pi\dot b_{x,n}=\dot b_{\pi(x,n)}$. This gives naturally rise to the sets $\dot B_x=\{\dot b_{x,n}\mid n<\omega\}^\bullet$. Since $\pi$ must be the identity on its $V_\alpha$-component, $\pi\dot B_x=\dot B_x$ for all $\pi$. It follows that indeed $\dot B=\{\dot b_{x,n}\mid x\in V_\alpha, n<\omega\}^\bullet$ can be mapped onto $V_\alpha$. We continue to work within the intermediate model given by $\IS_{\alpha.5}$ for just a bit longer while we investigate this symmetric system and its effects on the universe.

\begin{proposition}
$\forces^\HS_{\QQ_{\alpha,1}}\dot B$ is Dedekind-finite.
\end{proposition}
\begin{proof}
  Suppose that $\dot F\in\HS$ is a function and $t\forces_{\QQ_{\alpha,1}}\dot F\colon\check\omega\to\dot B$. Let $E$ be a finite set such that $\fix(E)\subseteq\sym(\dot F)$, we may assume without loss of generality that $\supp t\subseteq E$. Let $t'\leq t$ be such that for some $\tup{x,n}\notin E$ and $m<\omega$, $t'\forces\dot F(\check m)=\dot b_{x,n}$.

  Let $n'<\omega$ be such that $\tup{x,n'}\notin E\cup\supp t'$, and consider $\pi\in\sG_{\alpha,1}$ defined by $\pi(x,n)=\tup{x,n'}$, $\pi(x,n')=\tup{x,n}$, and $\pi=\id$ otherwise. Clearly, $\pi\in\fix(E)$ and $\pi t'$ is compatible with $t'$, but they force that the value of $\dot F(\check m)$ is different. This is impossible, therefore no such $\tup{x,n}$ exists, so in particular $\dot F$ must have a finite range and it is not injective.
\end{proof}

It remains to show that $\dot\QQ_{\alpha,1}$ did not add any new sets of rank $\alpha$, and since $\QQ_{\alpha,0}$ did not either, this will complete the proof. For this we need to first identify the names for the conditions of $\dot\QQ_{\alpha,1}$. We will work in the extension $\IS_\alpha$ for the remainder of this argument.

We can give $T$, the choice-tree, a canonical name. Each $\dot A_{\alpha,n}$ has a canonical $\dot\QQ_{\alpha,0}$-name, $\{\dot a_{\alpha,n,k}\mid k<\omega\}^\bullet$. For $f\in\omega^{<\omega}$ we let $\dot c_f=\tup{\dot a_{\alpha,n,f(n)}\mid n\in\dom f}^\bullet$, and $\dot T=\{\dot c_f\mid f\in\omega^{<\omega}\}^\bullet$. Given any finite $E\subseteq V_\alpha\times\omega$ and a sequence $p$ of the form $\tup{f_i\mid i\in E}$, where $f_i\in\omega^{<\omega}$, we define $\dot t_p=\tup{\dot c_{f_i}\mid i\in E}^\bullet$. From this we have that $\dot\QQ_{\alpha,1}$ has a natural name:\[\{\dot t_p\mid p\text{ a finite support function from }V_\alpha\times\omega\to\omega^{<\omega}\}^\bullet.\]

It is easy to see that this name is in $\HS$, since the only thing we need to verify is that if $\dot c_f$ is a name for a finite choice sequence from the $\dot A_{\alpha,n}$s, then $\pi\dot c_f$ is also such a name, but this is true by the virtue of $\pi\dot A_{\alpha,n}=\dot A_{\alpha,n}$ for all $n<\omega$. We can now talk about $\QQ_\alpha,\sG_\alpha$ and $\sF_\alpha$ as the iteration of $\tup{\QQ_{\alpha,0},\sG_{\alpha,0},\sF_{\alpha,0}}\ast\tup{\dot\QQ_{\alpha,1},\dot\sG_{\alpha,1},\sF_{\alpha,1}}^\bullet$.

\begin{theorem}\label{prop:morris-dist}
Suppose that $\dot a$ is a $\QQ_\alpha$-name which is symmetric and $\forces_{\QQ_\alpha}^\HS\dot a\subseteq\check V_\alpha$. Then there is some $a\subseteq V_\alpha$ and a condition $q=\tup{q_0,\dot q_1}$ such that $q\forces\dot a=\check a$.
\end{theorem}
\begin{proof}[Sketch of Proof]
  We may assume that every name that appears in $\dot a$ is of the form $\check x$ for some $x\in V_\alpha$. Next, let $E$ be a support for $\dot a$, that is $E=\tup{E_0,E_1}$ where $E_0\subseteq\omega$ and $E_1\subseteq V_\alpha\times\omega$ are finite.

  If $q=\tup{q_0,\dot q_1}$ is a condition in $\QQ_\alpha$, we will denote by $q\restriction E$ the condition \[\tup{q_0\restriction E_0\times\omega\times\kappa,\dot q_1\restriction E_1}.\] Suppose that $q\forces\check x\in\dot a$, then by the standard homogeneity arguments, $q\restriction E\forces\check x\in\dot a$, so we may assume that $q=q\restriction E$. The same applies for $\check x\notin\dot a$, of course. We can therefore assume that $\supp\dot q_1=E_1$, or extend it arbitrarily if necessary.

  Let $k$ be $\max E_0+1$, then we claim that we can further restrict $\dot q_1$ to any of its $k$-minimal weakening. Suppose that we proved this claim, then the theorem will follow: for any $x\in V_\alpha$ there is some $q'\leq q$ such that $q'\forces\check x\in\dot a$ or $q'\forces\check x\notin\dot a$, but now we can restrict $\dot q_1'$ to one of its $k$-minimal weakening. However, since $\dot q'_1\leq\dot q_1$, we can take this $k$-minimal weakening to be a weakening of $\dot q_1$ as well. It follows that $\dot a$ is reduced to a $\QQ_{\alpha,0}$-name given by \[\{\tup{r,\check x}\mid\exists\dot s\in\dot\QQ_{\alpha,1}\ k\text{-minimal}, \tup{r,\dot s}\forces_{\QQ_\alpha}\check x\in\dot a\}.\] Finally, since $\dot\QQ_{\alpha,0}$ did not add new subsets to $V_\alpha$, the conclusion follows.

  It remains to prove the claim. Namely, that if $\tup{q_0,\dot q_1}\forces\check x\in\dot a$, then $\tup{q_0,\dot q_1^{\min}}$ does so as well, where $\dot q_1^{\min}$ is a $k$-minimal weakening of $\dot q_1$. Suppose that $\dot s$ is such that $\dot s\leq_{\QQ_{\alpha,1}}\dot q_1^{\min}$ and $\supp\dot s=E_1$ as well, we will find an automorphism of $\QQ_{\alpha,0}$, $\pi$, such that:
\begin{enumerate}
\item $\pi\in\fix(E_0)$ which implies that $\pi q_0=q_0$.
\item $\pi\dot q_1^{\min}=\dot q_1^{\min}$.
\item $\pi\dot s$ is compatible with $\dot q_1$.
\end{enumerate}
Given that such a $\pi$ will not move $\dot a$ or $\check x$, it must be that no two extensions of $\dot q_1^{\min}$ can force contradictory truth values to $\check x\in\dot a$, which therefore completes the proof.

It would be helpful to think about the conditions of $\dot\QQ_{\alpha,1}$ as a finite sequence of stems in $\omega^{<\omega}$ in order to understand the situation. Above $\dot q_1^{\min}$ both $\dot q_1$ and $\dot s$ must have the same type in every level of the tree (above $k$). Namely, given $n\geq k$, the sequences which have the same value in the $n$th level of $\dot q_1$ must also have the same value in the $n$th level of $\dot s$. Moreover, if they do have the same value, this must be already captured by $\dot q_1^{\min}$.

So for each $n\geq k$ we can find a permutation of $\omega$, $\pi_n$, finite if we so desire, which moves the values of $\dot s$ in each level to the values of $\dot q_1$ without changing the values of $\dot q_1^{\min}$ in the level, if any exist. For all $n<k$, or large enough $n$ (above the levels in $\dot s$), let $\pi_n=\id$.\footnote{This is easier to visualise under the assumption that $\dot q_1^{\min}$ is literally just restricting all stems in $\dot q_1$ to their initial segment up to $k$, and that $\dot s$ is just an extension by a single level.}

We now let $\pi$ be the automorphism defined from the permutation of $\omega\times\omega$ given by the sequence $\tup{\pi_n\mid n<\omega}$. It is easy to see that $\pi$ satisfies the wanted properties since $E_0\subseteq k$ and $q_0=q_0\restriction E_0$. This completes the proof, as wanted.
\end{proof}
\section{Hic sunt dracones}\label{sec:pandemonium}
The revised Morris model is an impressive one. Indeed, one cannot extend it to a model of $\ZFC$ without adding ordinals,\footnote{This property was already true for the other constructions and was Morris' original goal.} and every set is the image of a Dedekind-finite set. But we can now utilise this general framework to tackle questions of saturation. Namely, if for every set $X$ there is a symmetric extension, or just any extension of the universe, in which $X$ has some counterexample to the axiom of choice, can we find a universe of set theory in which all the counterexamples already exist?

For example, the author's extension of L\"auchli's work.\footnote{L\"auchli used atoms and worked under the assumption that $F$ is countable.}
\begin{theorem*}[L\"auchli \cite{Lauchli:1963}, Karagila \cite{Karagila:MSc}]
  Given any field $F$, there is a symmetric extension in which $F$ has a vector space which is not finite dimensional, but every endomorphism of this vector space is a scalar multiplication.
\end{theorem*}

Or the theorem of Jech and Takahashi, both of whom proved independently the following theorem.

\begin{theorem*}[Jech \cite{Jech:ORD1966}, Takahashi \cite{Tahakashi:1968}]
  Given a partial order $(P,\preceq)$, there is a symmetric extension in which $P$ is realised as cardinals. That is, for every $p\in P$ there is a set $X_p$ such that $|X_p|\leq|X_q|$ if and only if $p\preceq q$.
\end{theorem*}

Or the theorem of Monro, related to Dedekind-finite sets.

\begin{theorem*}[Monro, \S3 of \cite{Monro:1975}]
Given any infinite cardinal $\kappa$, there is a symmetric extension in which $\kappa$ is the image of a Dedekind-finite set.
\end{theorem*}

In all of these examples we can easily construct a model where the conclusion holds for all the ground model sets. In the case of the Jech--Takahashi theorem, this was done by the author in \cite{Karagila:Embeddings}, and in the case of Monro's theorem, it was done by Monro himself in \S4 of the same paper. To some extent, Monro's original result and the author's \cite{Karagila:Morris} are in the same boat as these, where the author's work separate the load into a local version first to make this more explicit.

An obvious question, in all of these constructions, is whether or not we can saturate the universe with counterexamples. For example, given any field, is there already in our current universe a (non-trivial) vector space over the field whose endomorphisms are all scalar multiplications? Or, given a partial order, is it realised as cardinals in the model itself?\footnote{See \autoref{sec:oopsie} for some historical details and overview of this specific problem.} After all, just because we managed to capture all the sets from the ground model, it does not mean that we managed to capture the new sets. In the example of L\"auchli's result this is most striking. Whereas the embedding of partial orders into the cardinals and Monro's results can be argued to be of very ``pure set theoretic flavour'' and thus could possibly follow from some additional set theoretic axioms which may hold true in the models where all the well-orderable partial orders can be embedded into the cardinals, there is absolutely no reason to believe that such additional axioms can tell us anything about the existence of these odd vector spaces, or lack thereof.

Using the technique we described in this paper, these have a simple and positive answer. Simply iterate the symmetric systems where at each step we deal with all the fields, partial orders, what have you, which lie in $V_\alpha$ of the current model. Due to its iterative nature, we get more than just before, indeed each counterexample is added repeatedly, in all stages after it first appeared, so there is a proper class of ``different'' realisations of a partial order as cardinals; and over each field we have a proper class of pairwise non-isomorphic vector spaces, all of which satisfy the property ``every endomorphism is a scalar multiplication''.\footnote{To be technically correct, both of these will satisfy ``for every set of counterexamples, there is one that is not isomorphic to any of them'', since the term class refers to a definable class, and we simply do not know whether or not this collection of counterexamples is definable. We will use the term ``proper class of'', simply because it simplifies readability.}

We list here a few ``localised constructions'' for structural counterexamples. The idea is that we can combine them by taking the finite support product of these symmetric systems, where at each stage we add all counterexamples for all relevant objects which have rank $<\alpha$, using the relevant $\kappa$ at each stage. Using all the results from \autoref{sec:folk} we will immediately have that these counterexamples are all of different cardinality.

We will not prove all the details, but when things are not straightforward, we will sketch the ideas. We will rely on \autoref{thm:sym-copy} and will always assume that $\kappa$ used here is a suitable one.

\subsection{Partial orders as cardinals}
Given any partial order $(A,\leq)$ we want to realise it as cardinals, namely for every $p\in A$ we want to attach a set $X_a$ such that $|X_a|\leq|X_b|$ if and only if $a\leq b$. Since every partial order embeds into its power set ordered by inclusion,\footnote{$a\mapsto\{b\in A\mid b\leq a\}$.} we really want to embed $\power(A)$ into the cardinals.

Fix a set $A$, and consider $\PP$ as the symmetric copy of $A$ with the automorphism group $\{\id\}$ and $\sF=\{\{\id\}\}$. In other words, we are taking a product of $|A|$-copies of $\Add(\kappa,\kappa)$ with the standard Cohen-like symmetric system.

We let $\dot x_{a,\alpha}$ be the name $\{\tup{p,\check\beta}\mid p(a,\alpha,\beta)=1\}$ and we let $\dot X_a=\{\dot x_{a,\alpha}\mid\alpha<\kappa\}^\bullet$. Finally, for $B\subseteq A$, let $\dot X_B=\bigcup\{\dot X_a\mid a\in B\}$.

Standard arguments, such as those that we saw in \autoref{thm:incomp}, show that $\forces^\HS|\dot X_B|\leq|\dot X_C|\iff\check B\subseteq\check C$.\footnote{These actually extend to new, symmetric subsets of $A$, but we choose $\kappa$ such that no such subsets are added anyway.} Indeed, this is an immediate consequence of the theorem, since the products of different sets will have mutually generic copies.
\subsection{Non-isomorphic and non-existent algebraic closures}
This wonderful result, also due to L\"auchli in his seminal paper \cite{Lauchli:1963}, proved that the field of rational numbers can have two non-isomorphic algebraic closures. The result was further studied by Hodges \cite{Hodges:1976}.

The idea is simple: the algebraic closure of the rationals has a very rich automorphism group, so we can add a symmetric copy of it which cannot be well-ordered. In particular, this copy cannot be isomorphic to the canonical algebraic closure. For this strategy to work we need to make sure that there is no finite subset of the, or rather \textit{some}, algebraic closure such that when fixing this finite subset that algebraic closure is fixed pointwise.

In the context of $\ZFC$ this means that we can do it except when the field is real- or algebraically-closed to begin with. However, as Hodges proved in his paper, the symmetric copy added by L\"auchli is rigid. It has no automorphisms to begin with. Of course, it is algebraically closed, and it is a closure of the rationals which still have the well-ordered copy. It is not clear, however, if any extension of L\"auchli's closure (e.g.\ by adding $\pi$ to it) will have an algebraic closure and whether or not this closure will be rigid over our field.

Let $F$ be a field with an algebraic closure that is infinite dimensional and not rigid, and let $\overline F$ be an algebraic closure of $F$ witnessing that. We add a symmetric copy of $\overline F$ with the automorphism group $\aut(\overline F/F)$, and with the filter of groups generated by finite subsets. The symmetric copy obtained will have a distinct cardinality of any ground model algebraic closure.

At the same time, L\"auchli also showed that it is consistent that some fields do not have an algebraic closure. This is done by adding a symmetric copy of $\omega$, using $S_\omega$ as the automorphism group, and generating the filter by fixing pointwise finite subsets of $\omega$.

As usual, we define $\dot x_{n,\alpha}=\{\tup{p,\check\beta}\mid p(n,\alpha,\beta)=1\}$ and $\dot a_n=\{\dot x_{n,\alpha}\mid\alpha<\kappa\}^\bullet$. We let $\dot A=\{\dot a_n\mid n<\omega\}^\bullet$ be the symmetric copy of $\omega$. If $\dot F\in\HS$ is a name for an algebraically field of characteristic $\neq 2$, such that $\dot A\subseteq\dot F$ then there is some finite $E$ such that any $\pi\in\fix(E)$ preserves $\dot F$ and its operations, and therefore induces a field automorphism. To see that, note that \[\dot\pi=\{\tup{\pi p,\tup{\pi\dot x,\pi\dot y}^\bullet}\mid p\forces\dot x,\dot y\in\dot F,\dot x,\dot y\text{ appear in }\dot F\}\] is a name in $\HS$ for the induced automorphism.

Since the field is algebraically closed it has solutions to $x^2-1=0$, call them $i$ and $-i$, and assume that $\fix(E)$ also fixes these two pointwise.

We can now ignore the forcing and just work in the symmetric extension. Taking $n,m$ such that $\{n,m\}\times\lambda\cap E=\varnothing$ and setting $z\in F$ such that $z^2=a_n-a_m$. Now let $\pi$ be the $2$-cycle $(n\ m)$, then the induced automorphism satisfies $\pi z^2 = -z^2$ and moreover $\pi^2=\id$. From $\pi z^2=-z^2$ we have that either $\pi z=iz$ or $\pi z=-iz$. But we took $\pi i=i$, so if $\pi z=iz$, then $z=\pi\pi z=\pi(iz)=\pi i\pi z=i^2 z=-z$.

\subsection{Failure of weak choice axioms}
We write $\AC_X^Y$ to denote the statement ``Every family of sets indexed by $X$ each of whom has cardinality $Y$ admits a choice function''. To violate $\AC_X^Y$ we add a symmetric copy of $X\times Y$ with the automorphism group given by $\{\id\}\wr S_Y$, and fixing an ideal of subsets $I$ on $X$ which contains the finite subsets, we let $\sF_I$ be generated by $\fix(E)$ for $E\in I$, where $\fix(E)=\{\pi\in\{\id\}\wr S_Y\mid\pi\restriction E\times Y=\id\}$.

Then in the symmetric extension we have $\dot a_{x,y,\alpha}=\{\tup{p,\check\beta}\mid p(x,y,\alpha,\beta)=1\}$, then $\dot u_{x,y}=\{\dot a_{x,y,\alpha}\mid\alpha<\kappa\}^\bullet$ and $\dot U_x=\{\dot u_{x,y}\mid y\in Y\}^\bullet$. The family $\{\dot U_x\mid x\in X\}^\bullet$ is hereditarily symmetric, and indeed the bijection with $X$ given by $\tup{\dot U_x\mid x\in X}^\bullet$ is hereditarily symmetric.

However one can easily check that a subset family of $\{U_x\mid x\in X\}$ admits a choice function if and only if its domain is in $I$. Using \autoref{thm:incomp} we also get that the set $\{\dot u_{x,y}\mid\tup{x,y}\in X\times Y\}^\bullet$ will have a cardinality different from all the ground model's counterexamples to $\AC_X^Y$.

\subsection{Nonabelian cohomology of an infinite set is non-trivial}
This failure was studied by Andreas Blass in \cite{Blass:1983} where he shows that $H^1(X,G)$ is trivial for any group $G$ if and only if $\AC_X$, that is $\forall Y\,\AC_X^Y$, holds. Here we follow Blass' decision to treat $X$ as a discrete space and consider Giraud's definition of nonabelian cohomology,\footnote{Blass argues that one can use any nonabelian cohomology theory which satisfies a few basic properties, but argues that Giraud's definition is the natural choice.} which is better suited to a choice-free setting. We consider $G$-torsors over $X$, i.e.\ pairs $\tup{T,p}$ where $T$ is a set and $p\colon T\to X$ is a surjection whose fibres are exactly the orbit of the free action of $G$ on $T$. We say that two $G$-torsors are isomorphic if there is a bijection which commutes with the $G$-action and the projection. Here $H^1(X,G)$ is the class of isomorphism classes of $G$-torsors.\footnote{Each isomorphism class can be represented a set by applying Scott's trick.} This may also be the place to mention that Blass' work generally assumes that $H^1(X,G)$ is a set, but it really just asks whether or not it is trivial.

Since $X\times G$ with the natural projection is a $G$-torsor, $H^1(X,G)$ is never empty, but it could very well be trivial if any two $G$-torsors are isomorphic, which happens exactly when every $p\colon T\to X$ admits a section. If, on the other hand, $\AC_X^Y$ fails, then there is a family of sets $\{Y_x\mid x\in X\}$ such that for each $x\in X$, $Y_x$ has a bijection to $Y$, but there is no choice function from the family. Blass shows that in that case $H^1(X,S_Y)$ is non-trivial.

Seeing how the fibres of the surjection on $X$ are not uniformly equipotent to $Y$, and since $Y$ may not even carry a group structure,\footnote{The axiom of choice is in fact equivalent to the statement ``Every non-empty set admits a group structure''.} we need to use $S_Y$ instead. Consider for each $x$, $T_x$ as the set of bijections between $Y$ and $Y_x$, then $S_Y$ has a natural right action on $T_x$. This eliminates the need for choosing a bijection on each fibre, and if this $G$-torsor is indeed isomorphic to $X\times S_Y$, this defines a choice of bijections between $Y_x$ and $Y$ from which we can choose from the original family.

Note that iterating the construction for the failure of $\AC_X^Y$, as per \autoref{thm:incomp}, introduces new elements to $H^1(X,S_Y)$. More generally, $G$ acts on $X\times G$ in the natural way which makes it a $G$-torsor. Assuming that $G$ is infinite, a symmetric copy of $X\times G$ with $G$ as the automorphism group and stabilisers of finite subsets will satisfy the requirement of \autoref{thm:incomp}.

\subsection{Endormorphisms of vector spaces}
Suppose that $F$ is a given field, consider the $|X|$-dimensional vector space $F^{(X)}$, where $X$ is any infinite set. Add a symmetric copy of $F^{(X)}$ with $\aut(F^{(X)})$ as the automorphism group, and for some ideal of subspaces $I$, let $\sF_I$ be generated by pointwise stabilisers of vector spaces in $I$.\footnote{As usual, we will assume that all finite dimensional subspaces are in $I$.}

Let $K$ denote the symmetric copy of the vector space, then a subspace of $K$ is either $K$ itself, or it is isomorphic to one of the spaces in $I$. Moreover, assuming that the ideal is a proper ideal, if $T\colon F^{(X)}\to F^{(X)}$ is a linear operator, then $T$ has a symmetric copy if and only there is a $W\in I$ such that whenever $S\restriction W=\id$, $S\circ T=T\circ S$. It is not hard to see that if such $W$ exists, then $Tw\in W$ for all $w\in W$. We also note that $\ker T$ and $\operatorname{im}T$ are subspaces which are definable from $T$, so they must be in $I$ as well.

If $W$, and indeed every subspace in $I$, is such that for any $u,v$ which are linearly independent over $W$ there is some automorphism $S$ such that $Su=v$ and $S\restriction W=\id$, then we can show that $Tv\in W$ implies that $Tv=0$ for all $v\notin W$. From this we can conclude that $Tv=f\cdot v$ for some $f\in F$, although which $f$ may depend $v$ at this point. However, if $v,u$ are two vectors which are linearly independent over $W$ and $f_u,f_v,f_{u+v}$ are the three scalars matching the three vectors, then \[0=T(u+v-(u+v))=f_uu+f_vv-f_{u+v}(u+v)=(f_u-f_{u+v})u+(f_v-f_{u+v})v,\]and by linear independence we get that $f_u=f_{u+v}=f_u$. Finally, since $I$ is a proper ideal, $W$ does not have a complement, and therefore $V\setminus W$ must span the whole space, so $T$ is a scalar multiplication.

To find an example of an ideal where these conditions hold we need to look no further than taking $X=\omega$ and $I$ to be the finitely dimensional subspaces, which is generated by the subspaces spanned by a finite subset of the canonical countable basis.

\subsection{More general structures}
The examples above are just the tip of the iceberg in terms of what we can do and what can fail in the absence of $\AC$. Indeed, we have many theorems about rings, vector spaces, and we can vary the aforementioned ideals of subsets or subspaces to get all sort of examples for intermediate failures.

More broadly, we can say that a failure of the axiom of choice is \textit{structural} if it does not depend on the specific sets which encode it. So the failure of ``Every family of finite sets admits a choice function'' is a structural statement, whereas ``Every limit ordinal has countable cofinality'' or ``Every set of reals is Lebesgue measurable'' are not, since they talk about a specific type of sets. And while we can add a symmetric copy of the real numbers which will satisfy that every set is countable or contains a perfect subset, this will not be the real numbers themselves, so this is not quite what we want to do if we work in set theory or in a context where the failure is not structural.

In other words, $\varphi$ is a structural failure of the axiom of choice, if $\varphi$ can be falsified by adding a symmetric copy of a set \`a la \autoref{thm:sym-copy}. So we will refer to such a symmetric system as a structural one.

\subsection{Morris coding}
The above definition seems to exclude the Morris' model constructed in the previous section. Indeed, the structural part of the Morris model is adding an uncountable set which is a union of countably many countable sets.

We say that a function $F\colon V\to V$ is \textit{rank bounded} if there is some $\eta$ such that whenever $x\in V_\alpha$, $F(x)\in V_{\alpha+\eta}$.\footnote{For example, $F(x)=\power(x)$ is rank bounded, but $F(x)=V_{\alpha+\alpha}$, when $\alpha$ is the von Neumann rank of $x$, is not.}

\begin{definition}
  We say that a structural statement $\varphi$ is \textit{Morris codable} if for all $\alpha$ we can add a structural failure $A_\alpha$ such that for some rank bounded $F$
  \begin{enumerate}
  \item $F(A_\alpha)$ can be mapped onto $\omega_\alpha$ (or even $V_\alpha$).
  \item $\varphi$ proves that $A_\alpha\cong A_\beta$, for the appropriate notion of $\cong$.
  \end{enumerate}
\end{definition}

The previous section shows that ``countable union of countable sets is countable'' is Morris codable with $F(x)=\power(x)$.

\begin{corollary}
If $\varphi$ is a Morris codable statement, then it is consistent with $\ZF$ that ``$V\models\lnot\varphi$ and any outer model of $V$ in which $\varphi$ holds must have new ordinals.''\qed
\end{corollary}

It is not clear whether or not every structural failure is Morris codable, but we clearly see that statements of the form $\AC^Y_X$ are Morris codable when $X$ is infinite and $|Y|>1$ and both are well-orderable, since the argument is essentially the same as the one in the Morris model.

\subsection{Dragons}
Let $V$ be a model of $\ZFC$. We iterate symmetric systems as follows: in the $\alpha$th step we take the finite support product of all structural symmetric systems in $V_\alpha$ of the intermediate model so far, using the smallest suitable $\kappa$ possible; where appropriate we also iterate a symmetric system by its Morris coding, provided that the bound in the rank bounding is at most $\alpha+1$. Since the definition is uniform, the symmetric systems satisfy the conditions for the symmetric iteration.\footnote{This would not be the case if we were to try and be minimal about it, e.g.\ choose a single field of each isomorphism class and add a vector space over that field.}

Let $M$ be the model obtained by this iteration. In this model every structural failure repeats itself unboundedly often. Some of the statements that hold in this model are:
\begin{enumerate}
\item If $F$ is a field that has a non-rigid, infinite dimensional algebraic closure, then it has a proper class of non-isomorphic algebraic closures.\footnote{This, in particular, includes all ground model fields.}
\item If $F$ is a field, then $F$ has a vector space which has two linearly independent vectors, but there is no endomorphism of this vector space which is not scalar multiplication.
\item If $P$ is a partial order, then $P$ is realised by cardinals.
\item If $X$ and $Y$ are non-empty sets such that $X$ is infinite and $Y$ has at least two elements, then there is a proper class of non-equipotent sets which can be mapped onto $X$ with each fibre having exactly $|Y|$ elements, but there is no choice function from any infinitely many fibres. In particular, $H^1(X,S_Y)$ is a proper class for any $S_Y$. Indeed, for any reasonably interesting group $G$, repeatedly adding symmetric copies of $X\times G$ will result in $H^1(X,G)$ being a proper class for any infinite $G$ or $G$ of the form $S_Y$ for $|Y|>1$.
\item If $X$ is any set, then there is a Dedekind-finite set that can be mapped onto $X$.
\item If $X$ is any non-empty set, then there is a set $A$ which is a countable union of countable sets such that $\power(A)$ can be mapped onto $X$.
\item Every type of set which can be added via a symmetric extension (e.g., amorphous sets or a Dedekind-finite countable union of pairs) has a proper class of non-equipotent examples.
\end{enumerate}

And more generally, if $\varphi$ is a structural failure that can be forced by a symmetric system over $M$, then this symmetric system exists in some $V_\alpha$, and therefore by the $\alpha$th step we began introducing this failure to $M$. In particular $M$ is saturated, in the sense that any failure that can happen in some symmetric extension will already happen in $M$.

\section{Further research}
The model $M$ we describe is in some sense the ultimate model of $\lnot\AC$. However, it does miss a lot of the set theoretic failures which are quite interesting. Things such as successor cardinals which are singular, or that the real numbers are a countable union of countable sets, or that every uncountable set of reals has a regularity property (e.g., Lebesgue measurable, Baire property, the perfect set property). These are all interesting failures of the axiom of choice which need not occur in $M$. Indeed, we can insist to begin our construction above the reals so that they remain well-ordered, and any other set which is used in classical analysis remains well-ordered. In that sense, $M$ can be seen as a great separator between structural mathematics, which is the hallmark of modern approaches such as category theory, and the more ``hands-on'' mathematics that deals with specific structures such as analysis.

But we can still incorporate some failures into $M$ by beginning with a symmetric extension that turns the reals into a countable union of countable sets, for example, we can obtain this failure along with the rest of them. Perhaps most interestingly, since $L(\RR)$ is always a symmetric extension of a model of $\ZFC$, that means that we can incorporate the symmetric extension from some inner model $W$ as our first step. If $L(\RR)$ satisfies the Axiom of Determinacy, or $\AD$, then this will be captured by our structural iteration as well.

One structural failure of interest is that of Kinna--Wagner Principles. We discuss them in more length in \cite{Karagila:Iterations}, but in short we say that $\KWP_\alpha$ holds if for every set $X$ there is an ordinal $\eta$ such that $X$ injects into $\power^\alpha(\eta)$; we write $\KWP$ to mean $\exists\alpha\,\KWP_\alpha$. Note that $\KWP_0$ is just the axiom of choice, and that $\KWP_1$ implies that every set can be linearly ordered. It is not hard to find structural systems which violate $\KWP_n$ for any $n$ (simply consider $\omega^n$ with $S_\omega\wr\dots\wr S_\omega$ and finite supports). But it is not immediately obvious how to introduce higher failures of $\KWP_\alpha$. Of course, if that is at all possible, then $M$ will detect it and satisfy $\lnot\KWP$.

Another obvious refinement is Dependent Choice, or $\DC$. The statement $\DC_\kappa$ is ``every $\kappa$-closed tree has a chain of type $\kappa$'', and $\DC$ simply means $\DC_\omega$. We also write $\DC_{<\kappa}$ to read $\forall\lambda<\kappa,\DC_\lambda$. In the context of symmetric extensions, if $\PP$ is $\kappa$-distributive or $\kappa$-c.c.\ and $\sF$ is $\kappa$-complete, then $\DC_{<\kappa}$ is preserved (see \cite{Karagila:DC} and \cite{Banerjee:2020} for details), which means that in principle if a reasonable $\kappa$-support iteration framework for symmetric extensions is developed, then any structural failure compatible with $\DC_{<\kappa}$ can fail as well.

It should be noted that in this specific case this is likely to be fairly straightforward, since all the forcings can be taken as very closed. Nevertheless, the details need to be written down and checked before such claim can be verified.

Of course, this opens the door for questions about $\AC_\WO$, $\BPI$, and other choice principles. Although for these we do not have very good preservation theorems that make it easy to guarantee that a symmetric system will preserve them. So further work is sorely needed in these fronts. As well as the preservation of large cardinal axioms, which themselves have structural consequences on the universe. For example, if $\delta$ is a supercompact cardinal, then there is a generic extension which satisfies $\DC$. This means that any type of Morris coding which violates $\DC$ must destroy supercompactness. But perhaps it is possible to get some type of structural failures while preserving supercompactness, or at least some other large cardinal notions, nevertheless?

\subsection{Outside of set theory}
As we remarked in the case of algebraic closures, it is not entirely obvious that every field has an algebraic closure, and that every algebraic closure is finite dimensional over the base field or else it has a rich automorphism group. This raises some questions on this subject in $\ZF$.

\begin{question}
  Suppose that $F$ is a field which has an algebraic closure that is infinite dimensional. Does $F$ have an algebraic closure which has a rich automorphism group?
\end{question}

We also see that if we could add a symmetric copy of an algebraic closure once, this will happen a proper class of times. We can of course exclude a single field from the construction, e.g.\ only add symmetric copies of algebraic closures of fields that are not the rational numbers, or perhaps only add such algebraic closure during the first $\omega_1$ steps of the iteration.
\begin{question}
  Suppose that $F$ is a field that has two non-isomorphic algebraic closures. Is there a third? Are there infinitely many?
\end{question}

More generally, suppose that $\forall X\exists Y\psi(X,Y)$,\footnote{In the questions above, $\psi$ is ``the algebraic closure is either finite dimensional or has a rich automorphism group''. But this can be understood in many different ways.} is a structural statement. Under what conditions does it remain structural in $\ZF$? Moreover, when is a structural failure Morris' codable? Is there a structural criterion for that?

\begin{question}
  Suppose that $G$ is a finite group, then $G$ embeds into $S_G$. We know that $H^1(X,S_G)$ is a proper class. Can we use that to argue that $H^1(X,G)$ is a proper class as well?
\end{question}

\providecommand{\bysame}{\leavevmode\hbox to3em{\hrulefill}\thinspace}
\providecommand{\MR}{\relax\ifhmode\unskip\space\fi MR }
\providecommand{\MRhref}[2]{%
  \href{http://www.ams.org/mathscinet-getitem?mr=#1}{#2}
}
\providecommand{\href}[2]{#2}

\appendix
\section{History of realising partial orders as cardinals}\label{sec:oopsie}
Hartogs' theorem states that if every two cardinals are comparable, then the axiom of choice must be true. Therefore, if the axiom of choice fails, there are incomparable cardinals. Tarski proved that if there is a finite bound on the size of cardinal antichains, then the axiom of choice must be true.\footnote{This was published in Notices of the American Mathematical Society \textbf{11} (1964), 64T-348.} The result is relatively unknown, and so in \cite{FeldmanOrhonBlass:2008} it was reworked by David Feldman and Mehmet Orhon, who sent it to Andreas Blass who then offered a different argument, appearing in the appendix (Blass' argument is essentially that of Tarski's).

The natural question, of course, is how bad can the order of the cardinals get when the axiom of choice fails. Jech constructed a permutation model where given a fixed partial order we can embed it into the cardinals \cite{Jech:ORD1966}, using the transfer theorem from Jech--Sochor \cite{JechSochor:1,JechSochor:2} the result also holds in $\ZF$. Independently, Takahashi showed the same result, constructing a model of $\ZF$ directly using symmetric extensions \cite{Tahakashi:1968}.

Even later than that, Roguski showed that there is a class of pairwise incomparable cardinals \cite{Roguski:1990},\footnote{Much like our use in the paper, the result is that for every set $A$ of pairwise incomparable sets, there is a set $X$ such that $|X|\neq|a|$ for any $a\in A$.} although these results can be inferred from the work of Monro \cite{Monro:1975}.

The author's first publication \cite{Karagila:Embeddings} included improving slightly upon all of these results by embedding all the ground model partial orders and preserving $\DC_\kappa$ for any fixed $\kappa$. But the question of whether or not all partial orders can be embedded into the cardinals in the same model left unanswered there due to the missing methods which are covered in this very work.

More recently, Eilon Bilinsky informed us that he is developing a new technique for obtaining choiceless results and that he obtained a model in which all partial orders are realised \cite{Bilinsky}, as we also have obtained a similar result after developing the framework of iterating symmetric extensions.

However, while revising this paper we found out that the problem was in fact solved by Forti and Honsell in \cite{FortiHonsell:1985a} using atoms, and then in \cite{FortiHonsell:1985b} by considering what is essentially Monro's model in \cite{Monro:1975}. The two papers were never cited or referred to in any literature that we know of, prior to this mention. Furthermore, it seems that a small modification of their construction will also allow us to preserve $\DC_{<\kappa}$ for some $\kappa$.

All of the proofs mentioned here rely on the basic fact: $\power(A)$ is a universal partial order for any partial ordering on a subset of $A$. Therefore, it is enough to show that for any set $A$, there is a set $|A|$ of ``very incomparable'' cardinals, so that we can embed $\power(A)$ using these cardinals. Forti and Honsell also use a somewhat more ad-hoc argumentation based on the construction of the model.

Some of these papers also deal with the preorder $\leq^*$ which is defined by surjections, and we have certainly omitted some other papers which mention this problem or deal with questions about antichains of cardinals.
\bigskip
\end{document}